\documentclass{article}
\usepackage{graphicx} 

\usepackage{import}
\usepackage{shortcuts}
\usepackage{enumitem}
\import{./}{xlrboxes.tex}
\pgfplotsset{compat=1.18}

\newtheorem{theorem}{Theorem}[section]
\newtheorem{proposition}[theorem]{Proposition}
\newtheorem{lemma}[theorem]{Lemma}

\newtheorem{corollary}[theorem]{Corollary}

\theoremstyle{remark}
\newtheorem{remark}[theorem]{Remark}

\theoremstyle{definition}
\newtheorem{definition}[theorem]{Definition}
\newtheorem*{notation}{Notation}

\title{Power sum elements in the $G_2$ skein algebra}
\author{Bodie Beaumont-Gould
\and Erik Brodsky
\and Vijay Higgins
\and Alaina Hogan
\and Joseph M. Melby
\and Joshua Piazza}

\begin{document}
\date{}
\maketitle

\begin{abstract}
We study the skein algebras of surfaces associated to the exceptional Lie group $G_2,$ using Kuperberg webs. We identify two 2-variable polynomials, $P_n(x,y)$ and $Q_n(x,y),$ and use threading operations along knots to construct a family of central elements in the $G_2$ skein algebra of a surface, $\mathcal{S}_q^{G_2}(\Sigma),$ when the quantum parameter $q$ is a $2n\text{-th}$ root of unity. We verify these elements are central using elementary skein-theoretic arguments. We also obtain a result about the uniqueness of the so-called transparent polynomials $P_n$ and $Q_n.$ Our methods involve a detailed study of the skein modules of the annulus and the twice-marked annulus.
\end{abstract}

\tableofcontents

\section{Introduction}

The subject of skein theory lies at the intersection of low-dimensional topology and the representation theory of quantum groups. In skein theory, one works with linear relations involving pictures of skeins, which are hybrid objects that can be readily interpreted either in the language of topology or in the language of algebra. The simplest examples of these skein relations are the Kauffman bracket skein relations:

\begin{align*}
\xusebox{posCrossing}&=q \xusebox{4--Vlines} +q^{-1} \xusebox{4--Hlines}\\
\xusebox{unknot}&=-q^2-q^{-2}.
\end{align*}

On one hand, these relations can be used define link invariants, like the Jones polynomial, or to construct the 3-manifold invariants of Witten-Reshetikhin-Turaev \cite[Chapter 13]{Lick97}. On the other hand, these relations faithfully describe the representation theory of quantum $SL_2$; see e.g. \cite{Kup96}.

In this paper, we are interested in the skein theory associated to the rank 2 exceptional Lie group, $G_2.$ In this setting, skeins are certain graphs called \textit{webs.} The skein theory for $G_2$ webs, and for the cases of other rank 2 Lie groups, was developed in the seminal work of Kuperberg \cite{Kup94,Kup96}. Later, the skein theory for higher rank cases of $SL_d$ was developed in \cite{CKM14}. Finding explicit skein relations which faithfully describe the representation categories of Lie groups for all Dynkin types is an active area of research; see \cite{BERT21,SW22} for recent examples. Many of these skein theories are first developed in the setting where $q$ is a generic parameter. One can also study whether the same skein relations still faithfully describe, in some sense, the representation category when $q$ is root of unity; see \cite{Eli15,Bod22} for instance.

The setting of our work in this paper is of $G_2$ webs on surfaces, with particular focus on the case when $q$ is a root of unity. Two webs on a surface can be multiplied by stacking web diagrams on the surface, yielding an algebra structure. These skein algebras of surfaces are rich algebras related to both quantum groups and character varieties. Skein algebras are generally non-commutative, and their centers depend on the parameter $q.$ The main goal of this paper is to use the skein relations of Kuperberg to find central elements in the $G_2$ skein algebras.

\subsection{Power sum elements}

There is substantial literature on the subject of the $SL_2$ case of Kauffman bracket skein algebras of surfaces. In the root of unity case, surprising central elements were discovered in the work of Bonahon-Wong \cite{BW16}. These central elements arise from a topological operation of threading of certain polynomials along knots. The polynomials in question are the Chebyshev polynomials $T_n \in \mathbb{Z}[x]$ which satisfy the following fundamental property regarding the trace of a power of a matrix $A \in SL_2:$

\begin{equation*}
T_n(\text{tr}(A))=\text{tr}(A^n).
\end{equation*}

Since the trace of a power of a matrix is the power sum of the eigenvalues of the matrix, the polynomials $T_n$ can be referred to as power sum polynomials. These central elements obtained by threading power sum polynomials along knots may be referred to as power sum elements in the skein algebra. We are interested in finding analogous elements for the case of $G_2.$

The Chebyshev elements at roots of unity play a fundamental role in studying the representation theory of $SL_2$ skein algebras \cite{BW16,FKBL19,GJS19}. For closed surfaces, the Chebyshev elements generate the whole center of the skein algebra \cite{FKBL19}. We expect that the power sum elements found in this paper will play a similar role for the case of $G_2.$

Power sum elements have also been studied in the setting of a generic parameter $q$, and they satisfy product-to-sum formulas which allow for an explicit presentation of the skein algebra of the torus. They have been studied for each of the settings of the $SL_2$, HOMFLYPT, and Kauffman skein theories \cite{FG00,MS17,MPS23}. For the $SL_2$ case, Chebyshev elements are also involved in the construction of positive bases of skein algebras of surfaces, a construction which is important in the program of categorification of surface skein algebras and also is related to the theory of cluster algebras \cite{Thur14,Queff22,MQ23}.

\subsection{Main results}

We study the $G_2$ skein algebra, $\mathcal{S}_q^{G_2}(\Sigma),$ of a surface $\Sigma$ using the webs of Kuperberg. For our ring of coefficients, we use a commutative integral domain $\mathcal{R}$ containing a distinguished invertible element $q \in \mathcal{R}$. Due to denominators appearing in the defining skein relations, we further assume that $[12]^{-1} \in \mathcal{R},$ where the quantum integer $[k]$ denotes the Laurent polynomial defined by $[k]=\frac{q^{k}-q^{-k}}{q-q^{-1}}.$

In Section \ref{Pn,Qn}, we construct 2-variable polynomials $P_n(x,y)$ and $Q_n(x,y)$ in $\mathcal{R}[x,y]$ that are related to taking traces of powers of matrices in the two fundamental matrix representations of the Lie group $G_2$. We hence call elements obtained by threading $P_n$ or $Q_n$ along knots \textit{power sum elements} in the skein algebra. These elements satisfy a transparency property described in Section \ref{transparency section} and the polynomials $P_n$ and $Q_n$ provide examples of so-called \textit{transparent polynomials}.

\begin{theorem}(Theorem \ref{transparent})
If $q^{2n}=1$ then the element obtained by threading the polynomial $P_n$ or the polynomial $Q_n$ along a knot diagram on a surface $\Sigma$ is a central element of the $G_2$ skein algebra of the surface, $\mathcal{S}_q^{G_2}(\Sigma).$
\end{theorem}

We also obtain a result addressing the uniqueness of those polynomials which yield central elements by threadings.

\begin{theorem}(Theorem \ref{uniqueness})
Suppose that $q^2$ is a primitive root of unity of order $n$ coprime to $3.$ Then any transparent polynomial $S(x,y)$ is generated by $P_n$ and $Q_n.$ In other words, $S(x,y)$ must be an element of the polynomial ring $\mathcal{R}[P_n,Q_n].$
\end{theorem}

Along the way, we study the confluence theory of the $G_2$ skein relations (Theorem \ref{basis}) to describe bases of $G_2$ skein modules of surfaces. This had already been studied in \cite{Kup96} and \cite{SW07} using a smaller set of skein relations than the set that we use. We have incorporated the relations from the more recent works of \cite{SY17,BW21}. See Remark \ref{remark skein relations} comparing the two potential sets of skein relations.

The computations involved in establishing our main theorems take place inside of an algebra associated to the annulus $\mathcal{A}$, and an algebra associated to the twice-marked annulus $\mathcal{A}_{1,1}.$ These computations are inspired by the earlier work of L\^{e} in the $SL_2$ setting \cite{Le15}. Using our constructed bases, we show that $\mathcal{A} \cong \mathcal{R}[x,y]$ and that $\mathcal{A}_{1,1}$ is a commutative algebra, obtaining an explicit presentation for it in Corollary \ref{presentation A1,1}. These are essential ingredients for us in connecting the theory of commutative polynomials to the skein theory of $G_2$ webs.

A similar approach in constructing transparent polynomials using the explicit skein relations for the setting of $SL_d$ webs was carried out by Bonahon and the third author in \cite{BH23}. Informally, the strategy common between these papers is to attempt to match up certain elements in the commutative ring $\mathcal{R}[\lambda_1^{\pm 1},\lambda_2^{\pm 1}]$ with certain diagrammatic elements in $\mathcal{A}_{1,1}.$ In the $SL_d$ case, this strategy goes exactly as planned. When the analogous approach was tried in this project for $G_2$ webs, it produced a surprising error term. In Section \ref{error term} we discuss this error term and explain how to navigate around it.

The approach we use in this paper is an elementary approach using the explicit skein relations of Kuperberg. It would be quite interesting to extend this work to find a skein theoretic proof that the threading operation respects the Kuperberg web relations and describes an algebra homomorphism $S_1^{G_2}(\Sigma) \rightarrow S_q^{G_2}(\Sigma)$ when $q$ is an appropriate root of unity. There is an alternate approach to establishing the existence of these central elements, using the theory of quantum groups. It involves the Frobenius map of Lusztig at roots of unity; see \cite{GJS19}. In \cite{Kup96} it was shown that the $G_2$ webs and skein relations faithfully describe a full subcategory of the category of representations associated to the quantum group $G_2$ when the parameter $q$ is a generic parameter. It is predicted that the same still holds when $q$ is specialized to most roots of unity (see \cite{Mor11} and \cite{BW21}). Such a result should provide an avenue for the translation of results from the algebraic setting of quantum groups at roots of unity to the skein theoretic setting of Kuperberg webs. We expect that elements obtained by threading our polynomials $P_n$ and $Q_n$ along knots form a generating set for the image of the Frobenius map described in \cite{GJS19}.

\subsection{Acknowledgments}
This work took place during the SURIEM REU program at Michigan State University, hosted by the MSU Lymann Briggs College during Summer 2023. We are grateful for funding for this project provided by the NSA grant \# H98230-23-1-0005. The members of our project benefited from a synergistic environment resulting from a concurrently running undergraduate Summer Topology Program at MSU. We thank the organizers and participants of that program for stimulating conversations and for inviting us to daily tea times, funded by an NSF RTG grant DMS-2135960. We would also like to thank F. Bonahon for helpful conversations, as the idea for this project came out of a joint project between him and the third author.

\section{\texorpdfstring{$G_2$}{G2} skein relations and reduction rules}

In this section, we give the defining skein relations for $G_2$ skein algebras of surfaces. Let $\Sigma$ be an oriented surface of finite type. Let $\mathcal{R}$ denote a commutative integral domain containing an invertible element $q \in \mathcal{R}$ such that $[12]^{-1} \in \mathcal{R},$ where the quantum integer $[k]$ denotes the Laurent polynomial $[k]=\frac{q^{k}-q^{-k}}{q-q^{-1}}.$

\subsection{\texorpdfstring{$G_2$}{G2} skein algebras of surfaces}

An abstract $G_2$ web is a trivalent graph where each edge is either a single-strand or double-strand:

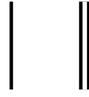
\begin{equation*}
\raisebox{-6pt}{\scalebox{1.5}{\xusebox{2--singleVline}}} \hspace{25pt} \raisebox{-6pt}{\scalebox{1.5}{\xusebox{2--doubleVline}}}
\end{equation*}

and each vertex is one of the following two types. We also allow vertexless loops of either strand type. The graph need not be connected.

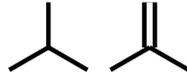
\begin{equation*}
\scalebox{2}{\xusebox{3--caltrop}} \hspace{7pt} \scalebox{2}{\xusebox{3--doubleCaltrop}}
\end{equation*}

We will study $G_2$ webs on surfaces subject to certain skein relations of Kuperberg. Linear combinations of webs provide the structure of an $\mathcal{R}\text{-module}$ called a skein module, and it is equipped with a natural algebra structure. We now provide the necessary definitions.

\begin{definition}
A \textit{web} $W$ on an oriented surface $\Sigma$ is a diagram on $\Sigma$ which is an embedding of a $G_2$ web except for transverse double-points which are labeled as overcrossings or undercrossings. Two webs $W_1$ and $W_2$ are isotopic on $\Sigma$ if there is an isotopy of $\Sigma$ taking $W_1$ to $W_2.$
\end{definition}

\begin{definition}
Let $\Sigma$ be a smooth oriented surface of finite type. We define the $G_2$ \textit{skein module} of $\Sigma$, denoted by $\mathcal{S}_q^{G_2}(\Sigma)$ to be the $\mathcal{R}\text{-module}$ spanned by isotopy classes of webs on $\Sigma$ subject to the skein relations (\ref{1crossing})-(\ref{pentagon}) listed later in this section.
\end{definition}

We now describe the algebra structure for the $G_2$ \textit{skein algebra} of a surface $\Sigma,$ also denoted by $\mathcal{S}_q^{G_2}(\Sigma).$

\begin{definition}
The module $\mathcal{S}_q^{G_2}(\Sigma)$ has an algebra structure arising from the bilinear extension of the following product of webs. If $W_1$ and $W_2$ are webs, then the product $W_1 \cdot W_2$ is obtained by isotoping the diagrams $W_1$ and $W_2$ on $\Sigma$ so that they only intersect at transverse double points and then taking the union $W_1 \cup W_2.$ At each intersection of $W_1$ with $W_2$, the crossing data is assigned so that $W_1$ is drawn above $W_2$ using over-crossings.
\end{definition}

The product operation on webs is well-defined with respect to isotopy because the defining skein relations imply the following types of Reidemeister moves, which hold for all strand types and vertex types:

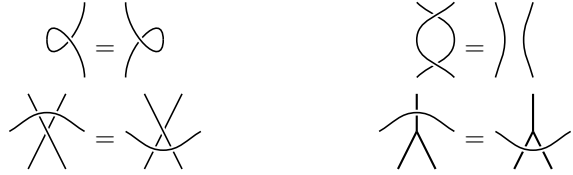
\begin{align*}
\begin{tikzpicture}[baseline=2ex]
\draw[thick, white, double=black, looseness=1.5] (0,1) to[out=-90, in=-90] (-.5,.5);
\draw[thick, white, double=black, looseness=1.5] (-.5,.5) to[out=90, in=90] (0,0);
\end{tikzpicture}&=\begin{tikzpicture}[baseline=2ex]
\draw[thick, white, double=black, looseness=1.5] (0,0) to[out=90, in=90] (.5,.5);
\draw[thick, white, double=black, looseness=1.5] (.5,.5) to[out=-90, in=-90] (0,1);
\end{tikzpicture} & \begin{tikzpicture}[baseline=2ex]
\draw[thick, white, double=black] (0,0) to[out=45, in=-90] (.5,.5) to[out=90, in=-45] (0,1);
\draw[thick, white, double=black] (.5,0) to[out=135,in=-90] (0,.5) to[out=90, in=-135] (.5,1);
\end{tikzpicture}&=
\begin{tikzpicture}[baseline=2ex]
\draw[thick, white, double=black] (0,0) to[out=75, in=-90] (.125,.5) to[out=90, in=-75] (0,1);
\draw[thick, white, double=black] (.5,0) to[out=105,in=-90] (.375,.5) to[out=90, in=-105] (.5,1);
\end{tikzpicture}\\
\begin{tikzpicture}[baseline=2ex]
\draw[thick, white, double=black] (0,0) -- (.5,1);
\draw[thick, white, double=black] (.5,0)--(0,1);
\draw[thick, white, double=black] (-.25,.5) to[out=30, in=180] (.25,.75) to[out=0, in=150] (.75,.5);
\end{tikzpicture}&=
\begin{tikzpicture}[baseline=2ex]
\draw[thick, white, double=black] (0,0) -- (.5,1);
\draw[thick, white, double=black] (.5,0)--(0,1);
\draw[thick, white, double=black] (-.25,.5) to[out=-30, in=180] (.25,.25) to[out=0, in=210] (.75,.5);
\end{tikzpicture}&\begin{tikzpicture}[baseline=2ex]
\draw[thick] (0,0) -- (.25,.5);
\draw[thick] (.5,0)--(.25,.5);
\draw[thick] (.25,.5)--(.25,1);
\draw[thick, white, double=black] (-.25,.5) to[out=30, in=180] (.25,.75) to[out=0, in=150] (.75,.5);
\end{tikzpicture}&=
\begin{tikzpicture}[baseline=2ex]
\draw[thick] (0,0) -- (.25,.5);
\draw[thick] (.5,0)--(.25,.5);
\draw[thick] (.25,.5)--(.25,1);
\draw[thick, white, double=black] (-.25,.5) to[out=-30, in=180] (.25,.25) to[out=0, in=210] (.75,.5);
\end{tikzpicture}
\end{align*}

In fact, there is an equivalent definition of the skein algebra of a surface, using the language of thickened surfaces and ribbon graphs, which satisfy the same Reidemester moves as above. More generally, one may define skein modules spanned by ribbon graphs in oriented 3-manifolds. In this paper, we focus on the more 2-dimensional perspective of webs on surfaces since that is the setting in which our computations take place.

The following skein relations serve two purposes. The first is that they are the defining relations for the skein modules. The second is that each relation can be interpreted as a reduction rule which replaces a web by a linear combination of simpler webs. We will soon explain how this leads to a basis for skein modules of surfaces, by using the method of confluence theory described in \cite{SW07}.

We use the following rules to remove crossings and internal double-edges.

\begin{align}
\xusebox{posCrossing}&=\frac{1}{[2]}\left(q^{3} \xusebox{4--Vlines} +q^{-3} \xusebox{4--Hlines}+q \xusebox{4--Hstick}+q^{-1} \xusebox{4--Vstick} \right) \label{1crossing}\\
\xusebox{posCrossDoubleTop}&= \frac{1}{[2]}\left(q^3 \xusebox{4--TRBLHstick} +q^{-3} \xusebox{4--TRBLVstick} + \frac{1}{[2]} \xusebox{4--2doubleSquare}   \right)\\
\xusebox{negCrossDoubleBot}&= \frac{1}{[2]}\left(q^{-3} \xusebox{4--TRBLHstick} +q^{3} \xusebox{4--TRBLVstick} + \frac{1}{[2]} \xusebox{4--2doubleSquare}   \right)\\
\xusebox{posCrossDoubleDouble}&= q^6 \xusebox{4--doubleVlines}+q^{-6} \xusebox{4--doubleHlines} + \frac{2}{[2]^2} \xusebox{4--4doubleSquare} + \frac{q^3}{[2]^3} \xusebox{4--4double2HSquare} + \frac{q^{-3}}{[2]^3} \xusebox{4--4double2VSquare}\\
\xusebox{DVstick}&= -\xusebox{4--Vlines} + \frac{[4][6]}{[2][12]} \xusebox{4--Hlines} + \xusebox{4--Hstick} + \frac{1}{[3]} \xusebox{4--Vstick}\label{double edge removal}
\end{align}

We call relation (\ref{double edge removal}) the `double-edge removal' relation. Once we have applied the crossing removals and double-edge removals to a web diagram on a surface, we are left with webs which might have some polygonal faces bounded by single-edges. The following defining reduction rules are used to remove small faces.

\begin{align}
\xusebox{unknot}&=\frac{[2][7][12]}{[4][6]} \text{ , } \xusebox{doubleunknot} = \frac{[7][8][15]}{[3][4][5]}\\
\xusebox{lollipop}&= 0 \text{ , } \hspace{1cm} \xusebox{doubleStrandLollipop}= 0\\
\xusebox{doubleLollipop}&= -\frac{[3][8]}{[4]} \hspace{1mm} \xusebox{2--singleVline} \text{ , } \hspace{1cm} \xusebox{doubleStrandDoublePop} = -[2]^2 \xusebox{2--doubleVline} \text{ , } \hspace{1cm} \xusebox{diffStrandDoublePop} = 0 \label{2gon}\\
\xusebox{spikyTriangle}&= \frac{[6]}{[2]} \xusebox{3--caltrop} \text{ , } \hspace{1cm} \xusebox{doubleEdgeTri} = -[3]\xusebox{3--doubleCaltrop} \text{ , } \hspace{1cm} \xusebox{twoDoubleEdgeTri} = 0 \label{3gon}\\
\xusebox{spikySquare}&= [3]\left(\xusebox{4--Vlines} + \xusebox{4--Hlines}\right) -\frac{[4]}{[2]}\left(\xusebox{4--Hstick} + \xusebox{4--Vstick}\right) \\
\xusebox{doubleEdgeSquare}&= \xusebox{4--TdoubleRopen} + \xusebox{4--LdoubleBopen} \\
\xusebox{spikyPentagon} &= \left(\xusebox{5--BBpent}+\xusebox{5--BRpent}+\xusebox{5--TRpent}+\xusebox{5--TLpent}+\xusebox{5--BLpent}\right) \label{pentagon}\\
     &\hspace{1.6cm} - \left(\xusebox{5--TTvert}+\xusebox{5--TLvert}+\xusebox{5--BLvert}+\xusebox{5--BRvert}+\xusebox{5--TRvert}\right) \nonumber
\end{align}

All of these defining relations are local, in the sense that they hold in any disk on the surface. Some of the defining relations follow from the other defining relations. See \cite{BW21,SY17} for a smaller set of defining relations. We have listed the larger set because these are needed to form a set of relations which are confluent, in the following sense.

We can view each relation (\ref{1crossing})-(\ref{pentagon}) as a reduction rule for rewriting a web as a linear combination of webs which are simpler when considering the triple of numbers of (crossings, vertices, connected components), considered lexigocraphically. This allows us to find a spanning set of webs for $\mathcal{S}_q^{G_2}(\Sigma)$ consisting of the set of webs which cannot be reduced further. In order for this spanning set to be a basis, \cite{SW07} explains that we must check that every overlap of the reduction rules can be resolved consistently. To this end, in addition to the defining relations, we need the following reduction rules. Even though these apply in annular regions instead of disk-shaped regions, they follow from the previous local rules. In the following pictures, the dot represents a region on the surface which may be non-contractible.

\begin{align}
\xusebox{Fake4Gon} &=  \xusebox{Hole2Loop}  -[2]^2 \xusebox{c_HoleDoubLoop} + \frac{[8]}{[4]} \xusebox{c_HoleLoop} - [7] \xusebox{Hole} \label{fake square}\\
\xusebox{Fake5Gon} &= \xusebox{Hole2Loop1} - \frac{[6]}{[2][3]} \xusebox{HoleLoop1} \label{fake pentagon}
\end{align}

The relation (\ref{fake square}) is derived by considering an overlap of the middle relation of (\ref{2gon}) and the relation (\ref{double edge removal}). The relation (\ref{fake pentagon}) is derived by considering an overlap of the last relation of (\ref{3gon}) and of (\ref{double edge removal}). We refer to the region bounded by the web in the left side of (\ref{fake square}) as a \textit{self-folded square,} and the region bounded by the web in the left side of (\ref{fake pentagon}) a \textit{self-folded pentagon.}

By considering an overlap of the relation (\ref{fake pentagon}) with itself we find we need the following equivalence rule (\ref{loop switch}). We do not consider it as a reduction rule, but we use it as an equivalence relation placed on the set of irreducible diagrams.

\begin{align}
\xusebox{singleDouble} &= \xusebox{doubleSingle} \label{loop switch}    
\end{align}

\subsection{Bases of skein modules}

When we apply a reduction rule, we replace a web with a linear combination of simpler descendant diagrams. Our relations are confluent if, whenever two reduction rules are applicable to a web, the element obtained by applying the first reduction rule and the element obtained by applying the second can both be reduced to the same element, up to surface isotopy and moves of the form (\ref{loop switch}). A set of reduction rules is terminal if no diagram has an infinite chain of descendants. If the defining set of relations for a module consists of a set of reduction rules which is terminal and confluent, then the irreducible diagrams form a linearly independent spanning set for the module. We encourage the reader unfamiliar with confluence theory (also called rewriting theory or the Diamond Lemma) to consult \cite{SW07} for the description of this procedure in the context of skein theory. Other similar examples of the process of checking the confluence of skein relations can be found in \cite{Le18,FS22,Hig23}.

\begin{theorem}\label{basis}
The above defining skein relations (\ref{1crossing})-(\ref{fake pentagon}) are terminal and confluent. Thus, a basis for $\mathcal{S}_q^{G_2}(\Sigma)$ consists of irreducible webs. These irreducible webs are characterized as having no crossings, no internal double-edges and no $n\text{-gon}$ (contractible) faces for $n \leq 5.$ Any two irreducible webs representing the same basis element are related by an isotopy of the surface along with a sequence of loop switch moves (\ref{loop switch}).
\end{theorem}

\begin{proof}
Note that each reduction rule (\ref{1crossing})-(\ref{fake pentagon}) strictly decreases the numbers of the following, in lexicographic order (crossings, vertices, connected components). Thus, the set of reduction rules is terminal. Then the proof follows directly from the methods of \cite{SW07}, by checking that the relations are locally confluent for all non-trivial overlaps of the reduction rules.
\end{proof}

\begin{remark}\label{remark skein relations}
If one considers only webs with single-strands and uses relations only concerning single-strands, then the relations appearing in \cite{Kup96} are complete, and from \cite{SW07} a basis of that version of the skein module can be deduced by leaving out the double-edge removal relation. However, we need the extra defining reduction rules along with (\ref{fake square}) and (\ref{loop switch}) to obtain a confluent presentation of the version of the skein module that uses both kinds of strands.
\end{remark}

We will also consider skein modules of surfaces with marked points on the boundary. In this paper, however, the only marked surface we explicitly deal with is the twice-marked annulus $\mathcal{A}_{1,1}$ as described in Section \ref{marked annulus}.

\begin{definition}
Let $M=p_1,...,p_k \in \partial \Sigma$ be a finite set of marked points on the boundary of $\Sigma.$ The skein module of the marked surface $\mathcal{S}_q^{G_2}(\Sigma,M)$ is the $\mathcal{R}\text{-module}$ spanned by webs $W$ having single-strand endpoints such that $\partial W=M.$ Webs are considered up to isotopies of $\Sigma$ which fix the boundary components.
\end{definition}

Our Theorem \ref{basis} also applies when $(\Sigma, M)$ is a surface with marked boundary. Following Kuperberg \cite{Kup94}, we can identify irreducible webs by analyzing contributions of web faces to the Euler characteristic of the surface.

\begin{definition}\label{Euler measure}(Euler measure)
Let $W$ be a crossingless web on a surface $\Sigma$ of finite type. We call a connected component of $\Sigma \setminus W$ a \textit{face} of $W$. Each face $U$ of $W$ is bounded by some number of edges and vertices. The edges could be edges (or loops) of $W$ or could be boundary edges (or boundary circles) of $\Sigma,$ if $W$ has endpoints on the boundary. If every edge bounding $U$ is an edge of $W,$ then $U$ is a \textit{interior face} of $W$. If at least one edge bounding $U$ is a boundary edge of $\Sigma,$ then we call $U$ a \textit{boundary face} of $W$.

To each face $U$ of $W,$ we define its \textit{Euler measure} $\mu(U)$ to be

\begin{equation*}
\mu(U)=\frac{1}{3}V+\frac{1}{2}V_\partial - \frac{1}{2}E-E_\partial + \chi(U),   
\end{equation*}

where $V$ and $V_\partial$ are the number of internal vertices and boundary vertices of $W,$ and $E$ is the number of edges of $U$ which belong to $W$ and $E_\partial$ is the number of edges of $U$ which belong to $\partial \Sigma.$ The quantity $\chi(U)$ is the ordinary Euler characteristic of the face $U$ considered as a subsmanifold of the surface $\Sigma.$ Loops of $W$ and boundary circles of $\partial \Sigma$ do not contribute to the Euler measure.
\end{definition}

A classical fact is that the Euler measure satisfies the following fundamental property. If the faces of $W$ are labeled by $U_1,...,U_k,$ then we have

\begin{equation*}
\sum_{i=1}^k \mu(U_i)=\chi(\Sigma).
\end{equation*}

\begin{proposition}
A web $W$ on $\Sigma$ is irreducible if and only if it contains no crossings, no internal double edges, and if every internal face $U$ of $W$ satisfies $\mu(U)\leq 0.$
\end{proposition}

Kuperberg used the notion of Euler measure to prove the following.

\begin{proposition} (\cite{Kup94})
If $\Sigma$ is a disk or is a sphere, then $\mathcal{S}_q^{G_2}(\Sigma) \cong \mathcal{R}.$
\end{proposition}

\begin{proof}
It suffices to show that our basis of $\mathcal{S}_q^{G_2}(\Sigma)$ consists of only the empty diagram. Suppose that $W$ is a basis web for $\mathcal{S}_q^{G_2}(\Sigma).$ Then $W$ contains no internal double-edges, no crossings, and no $n\text{-gons}$ for $n \leq 5.$ Since $\Sigma$ is a sphere or is a disk, the Euler characteristic of $\Sigma$ is positive. If $W$ is non-empty, then it must bound some region of positive Euler measure. However, a region of positive Euler measure must be an $n\text{-gon}$ with $n\leq 5.$ Since $W$ has no such face, $W$ must be empty.
\end{proof}

\section{Annuli}

We now study some algebras associated to skein modules of the annulus. In this section, we study these algebras from a purely skein-theoretic perspective. Later on, we will return to these algebras to perform computations with power sum polynomials.

\subsection{The annulus and the algebra \texorpdfstring{$\mathcal{A}$}{A}}

Let $\mathbb{A}=S^1 \times [0,1]$ denote the annulus. 

\begin{definition}
Define $\mathcal{A}$ to be the $\mathcal{R}\text{-algebra}$ spanned by closed webs in $\mathbb{A}$. If $W_1,W_2 \in \mathcal{A}$ are two webs, we define their product, $W_1W_2$ to be the web obtained by isotoping $W_1$ so it is contained in $S^1 \times (1/2,1]$, isotoping $W_2$ so it is contained in $S^1 \times[0,1/2),$ and taking the union of the webs. This product agrees with that of $\mathcal{S}_q^{G_2}(\mathbb{A}).$
\end{definition}

The multiplicative unit $1 \in \mathcal{A}$ is the empty diagram. We define $x,y \in \mathcal{A}$ to be the following distinguished elements.

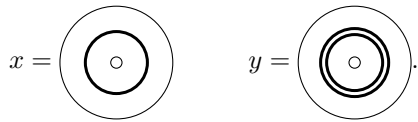
\begin{equation*}
x=\xusebox{x} \hspace{1cm} y=\xusebox{y}.
\end{equation*}

\begin{proposition}\label{A isomorphic}
The algebra $\mathcal{A}$ associated to the annulus is isomorphic to $\mathcal{R}[x,y]$ and has a basis of diagrams consisting of concentric loops around the core of the annulus.
\end{proposition}

\begin{proof}
We first describe the basis of $\mathcal{A}.$ Let $W$ be a basis web. Thus, $W$ contains no internal double-edge, and bounds no $n\text{-gon}$ for $n \leq 5.$ Consequently, $W$ bounds no region of positive Euler measure. If $W$ has a trivalent vertex, then as $W$ must bound a non-simply connected region on the annulus, $W$ would have a region of negative Euler measure. Therefore, as the Euler characteristic of $\mathbb{A}$ is zero, $W$ cannot contain a trivalent vertex. Thus, $W$ consists of concentric loops around the core of the annulus, and the description of the basis is as claimed.

Now we describe the algebra structure of $\mathcal{A}.$ Due to the equivalence rule (\ref{loop switch}), the two kinds of loops in $\mathcal{A}$ commute. Consequently, there is a well-defined algebra homomorphism $\mathcal{R}[x,y] \rightarrow \mathcal{A}$ sending $x$ to an essential single-strand loop and sending $y$ to an essential double-strand loop. This map sends the standard basis of monomials of $\mathcal{R}[x,y]$ injectively onto the basis we have just described for $\mathcal{A}.$ Thus, this algebra map is an isomorphism $\mathcal{R}[x,y] \cong \mathcal{A}.$
\end{proof}

\subsection{The twice-marked annulus and the algebra \texorpdfstring{$\mathcal{A}_{1,1}$}{A1,1}}\label{marked annulus}

Consider the annulus $\mathbb{A}=S^1\times [0,1]$ and a fixed point $p \in S^1$. Let $p_0=(p,0)$ and $p_1=(p,1)$ be two marked points on $\partial \mathbb{A}.$ We consider webs with boundary equal to $M=\{p_0,p_1\}.$

\begin{definition}
Define $\mathcal{A}_{1,1}$ to be the skein module $\mathcal{S}_q^{G_2}(\mathbb{A},\{p_0,p_1\})$ of the marked annulus spanned by webs with one single-strand endpoint at $p_0$ on the interior boundary of the annulus and one single-strand endpoint at $p_1$ on the exterior boundary of the annulus. The algebra structure for $\mathcal{A}_{1,1}$ is given by concatenating annuli outward, as described for $\mathcal{A}.$ Note that this product is not the same as a skein algebra product.
\end{definition}

\begin{remark}In the literature, $\mathcal{A}_{1,1}$ is also known as an endomorphism space in the \textit{annular representation category}, \textit{annular spider}, \textit{skein category of the annulus}, or as a corner of the \textit{tube algebra}. The product structure is to be thought of as the composition of morphisms. We will later show that $\mathcal{A}_{1,1}$ is commutative (Corollary \ref{A11 commutative}).
\end{remark}

The multiplicative unit of $\mathcal{A}_{1,1}$ is the following diagram.

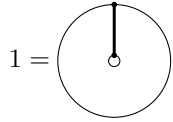
\begin{equation*}
1=\xusebox{A11--id}
\end{equation*}

We identify the following distinguished elements $a,a^{-1},c,f \in \mathcal{A}_{1,1}:$

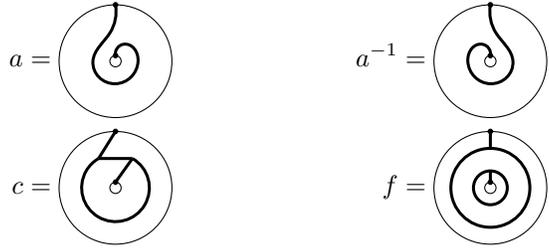
\begin{align*}
a&=\xusebox{A11--a} & a^{-1}&=\xusebox{A11--aInv}\\
c&=\xusebox{A11--c} & f&=\xusebox{A11--f}.
\end{align*}

In addition to the fundamental elements $a,a^{-1},1,c,f \in \mathcal{A}_{1,1}$ listed, we also need to provide notation for the elements $f_{i,j}\in \mathcal{A}_{1,1}$ which consist of the diagram $f$ along with $i$ concentric single-strand loops and $j$ concentric double-strand loops inside of the open annular region bounded by the diagram of $f.$

For example, the following depicts the element $f_{1,1}.$

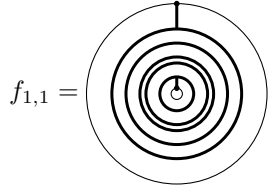
\begin{equation*}
f_{1,1}=\xusebox{f11}.
\end{equation*}

Using these elements, we can describe a basis for $\mathcal{A}_{1,1}$ over $\mathcal{R}$ (Proposition \ref{A1,1 basis}) and later a presentation for $\mathcal{A}_{1,1}$ as an $\mathcal{R}\text{-algebra}$ (Corollary \ref{presentation A1,1}).

We first record some topological facts about diagrams in $\mathcal{A}_{1,1}.$ Recall that the mapping class group of the annulus, $\text{MCG}(\mathbb{A})$, is the group of isotopy classes of orientation preserving diffeomorphisms of $\mathbb{A}$ which fix the boundary. It is a classical fact (see \cite{FM12}) that $\text{MCG}(\mathbb{A})\cong \mathbb{Z}$ and is generated by a Dehn twist of the annulus. The annulus is equipped with a natural orientation-reversing diffeomorphism $\mathbb{A} \rightarrow \mathbb{A}$, called the mirror image, which reverses the orientation of the $S^1$ factor of $S^1 \times [0,1].$ From these facts we observe that if a web $W_1 \in \mathcal{A}_{1,1}$ is taken to a web $W_2 \in \mathcal{A}_{1,1}$ by an orientation-preserving diffeomorphism, then $W_1=a^kW_2$ for some integer power $k$ of the twist diagram $a.$ Similarly, if $W_1$ is taken to $W_2$ by an orientation-reversing diffeomorphism, then $W_1$ and the mirror image of $W_2$ are related by a power of $a.$ We use these observations to record the following.

\begin{lemma}\label{MCG}
Suppose $W_1 \in \mathcal{A}_{1,1}$ is a web and $W_2\in \mathcal{A}_{1,1}$ is one of $\{1,c^j,f_{ij}\}.$ If $W_1$ and $W_2$ are diffeomorphic then there exists some $k \in \mathbb{Z}$ such that

\begin{equation*}
W_1=a^kW_2.
\end{equation*}
Moreover, if $W_2=f_{ij},$ then we may use $k=0.$
\end{lemma}

\begin{proof}
The webs $1$ and $f_{ij}$ are invariant under the mirror-image operation, so the statement follows in these cases from the previous discussion. Morover, it can be seen diagrammatically that $a^kf_{ij}=f_{ij}$ for all $k.$ Now consider the case $W_2=c^j.$ It is an easy observation that the mirror image of the diagram $c$ is equal to $a^{-1}c.$ From this, and the fact that the mirror image operation induces an algebra homomorphism $\mathcal{A}_{1,1}\rightarrow \mathcal{A}_{1,1},$ it follows that the mirror image of $c^j$ is equal to $a^{-j}c^j.$ Thus, $W_1$ is related by a power of $a$ to the diagram $a^{-j}c^j,$ as required.
\end{proof}

\begin{proposition}\label{A1,1 basis}
The algebra $\mathcal{A}_{1,1}$ associated to the twice-marked annulus has a basis of irreducible webs constructed in Theorem \ref{basis}. The basis consists of the following set of webs: $\{a^ic^j \mid i \in \mathbb{Z}, j \in \mathbb{Z}_{\geq 0}\} \cup \{f_{i,j} \mid i,j \in \mathbb{Z}_{\geq 0}\}$.
\end{proposition}

The proof of Proposition \ref{A1,1 basis} follows from Lemma \ref{set subset basis} and Lemma \ref{basis subset set}.

\begin{lemma}\label{set subset basis}
The elements of $\{a^ic^j \mid i \in \mathbb{Z}, j \in \mathbb{Z}_{\geq 0}\} \cup \{f_{i,j} \mid i,j \in \mathbb{Z}_{\geq 0}\}$ are distinct and each element of the set is some basis element of $\mathcal{A}_{1,1}.$
\end{lemma}

\begin{proof}
Since the elements of $\{a^ic^j\} \cup \{f_{i,j}\}$ are all irreducible, they are basis elements. To show they are distinct, we must show that they are not related to each other by sequences of isotopies and moves of the form $(\ref{loop switch}).$ Since isotopies and loop switch moves $(\ref{loop switch})$ preserve the number of single-strand and double-strand loops, we see that the set $\{f_{i,j}\}$ consists of distinct elements. Since every element of $\{f_{i,j}\}$ is disconnected while every element of $\{a^ic^j\}$ is connected, we see that $\{f_{i,j}\} \cap \{a^ic^j\} =\emptyset.$

We now show the elements of $\{a^ic^j\}$ are distinct. If $j\neq j',$ then $a^ic^j \neq a^ic^{j'}$ since the two diagrams have different numbers of vertices. If $i\neq i',$ then $a^i$ is not isotopic to $a^{i'}$ since the diagrams have different winding numbers. Now suppose that $a^ic^j$ is isotopic to $a^{i'}c^{j}.$ Within the underlying web of $c^j,$ there is a path from $p_0$ to $p_1$ which is isotopic rel endpoints to the identity element $1 \in \mathcal{A}_{1,1}.$ Concatenating this path with $a^i$ yields a path $\gamma$ in $a^ic^j$ isotopic rel endpoints to $a^i.$ On the other hand, if we concatenate the same path in $c^j$ with $a^{i'},$ we find a path $\gamma'$ in $a^{i'}c^j$ isotopic rel endpoints to $a^{i'}.$ An isotopy of $\mathbb{A}_{1,1}$ taking $a^{i'}c^j$ to $a^{i}c^j$ must take $\gamma$ to $\gamma'$, so we must have $i=i'.$ Therefore, elements of the set $\{a^ic^j\}$ are distinct.
\end{proof}

\begin{lemma}\label{basis subset set}
If $W$ is a basis web of $\mathcal{A}_{1,1},$ then $W$ must be an element of

\begin{equation*}
\{a^ic^j \mid i \in \mathbb{Z}, j \in \mathbb{Z}_{\geq 0}\} \cup \{f_{i,j} \mid i,j \in \mathbb{Z}_{\geq 0}\}.
\end{equation*}
\end{lemma}

\begin{proof}
Let $W$ be a basis web. We consider two cases. If $p_0$ and $p_1$ belong to different connected components of $W$, we claim that $W=f_{i,j}$ for some $i,j.$ Let $W_0$ be the connected component of $W$ containing the endpoint $p_0$ and let $W_1$ be the connected component of $W$ containing the endpoint $p_1.$ It will suffice to show the following equality of webs: $W_0 \cup W_1=f.$ Since $W_0$ has an endpoint at $p_0$ but not at $p_1,$ the web $W_0$ must contain a trivalent vertex. Consider the face of $\mathbb{A} \setminus W_0$ containing $p_1$. It is not simply connected, and so must have negative Euler measure. Thus, some other face of $W_0$ must have positive Euler measure. Since $W_0$ is irreducible, this face must be adjacent to the boundary. By visual inspection, this boundary face must be a self-folded square. Hence, $W_0$ is of the form claimed. The same argument shows that $W_1$ is of the form claimed and we must have the equality of diagrams $W_0 \cup W_1=f.$ Since the region between $W_0$ and $W_1$ is an annulus, any irreducible diagram in that region must be a collection of concentric loops, so $W$ must be equal to $f_{i,j}$ for some $i,j.$

For the next case, if $W$ contains a path from $p_0$ to $p_1$, we claim $W=a^ic^j$ for some $i \in \mathbb{Z}, j \in \mathbb{Z}_{\geq 0}.$ By Lemma \ref{MCG} it suffices to show that $W$ is diffeomorphic to $c^j$ for some $j \geq 0.$ Let $W'$ be the connected component of $W$ containing $p_0$ and $p_1$. If $W'$ contains no trivalent vertex, then $W'$ is diffeomorphic to $1=c^0.$ Otherwise, $W'$ contains a trivalent vertex and, like above, we consider the boundary face of $W'$ adjacent to $p_0$. It must be either a self-folded square or a self-folded pentagon. If it is a self-folded square, then $W'$ is the same as $W_0$ from the last paragraph. But $W_0$ did not connect $p_0$ to $p_1.$ Therefore, this boundary face must be a self-folded pentagon. Similarly, the boundary face adjacent to $x_1$ must be a self-folded pentagon as well. Since these boundary faces account for an Euler measure of zero, any other face of $W'$ must also have Euler measure zero. Therefore, all other faces of $W'$ must be hexagons. By inspection, they must be self-folded hexagons and so the diagram $W'$ must be diffeomorphic to some power of $c.$ Since all faces of $W'$ are topologically disks, any other component of $W$ would be contained in the complement of $W'$ and would thus be reducible. So we must have $W=W'=a^ic^j$ as claimed.
\end{proof}

\begin{corollary}\label{A11 commutative}
The algebra $\mathcal{A}_{1,1}$ is commutative.
\end{corollary}

\begin{proof}
By the characterization of basis elements of $\mathcal{A}_{1,1}$ in Proposition \ref{A1,1 basis}, one only needs to check that the elements $a,c,f_{i,j}$ commute with each other. Some of these are checked by simple isotopies while others are checked using short computations similar to the ones in the proof of the next corollary.
\end{proof}

\begin{corollary}\label{presentation A1,1}
The algebra $\mathcal{A}_{1,1}$ has the following presentation as an $\mathcal{R}\text{-algebra}$. It is the quotient of the algebra generated by the commuting variables $a,a^{-1},c,f_{ij}$ for $i,j \geq 0$ subject to the relations
\begin{align*}
af_{i,j}&=f_{i,j}\\
a^{-1}f_{i,j}&=f_{i,j}\\
f_{i,j}f_{k,l}&=f_{i+k+2,j+l}-[2]^2f_{i+k,j+l+1}+\frac{[8]}{[4]}f_{i+k+1,j+l}-[7]f_{i+k,j+l}\\
cf_{i,j}&=f_{i+1,j}-\frac{[6]}{[2][3]}f_{i,j}.\\
\end{align*}

\end{corollary}

\begin{proof}
We first suppose the relations above hold. Since each basis element of our basis for $\mathcal{A}_{1,1}$ can be written as a product of these proposed generators, our list is a generating set for $\mathcal{A}_{1,1}$. On the other hand, any element written as a linear combination of products of these generators can be reduced to a linear combination of basis elements using only the relations listed in the proposition. Therefore, if the relations hold then they are a full set of relations for $\mathcal{A}_{1,1}.$

We now check that the relations hold. The relation $af_{i,j}=f_{i,j}=a^{-1}f_{i,j}$ follows from observing a simple isotopy. The relation which rewrites $f_{i,j}f_{k,l}$ follows from the computation $f^2=f_{2,0}-[2]^2f_{0,1}+\frac{[8]}{[4]}f_{1,0}-[7]f,$ using the self-folded square relation (\ref{fake square}).
The relation which rewrites $cf_{i,j}$ follows from the computation $cf=f_{1,0}-\frac{[6]}{[2][3]}f,$ using the self-folded pentagon relation (\ref{fake pentagon}). 
\end{proof}

The following will be used later on to identify a nicely behaved subalgebra of $\mathcal{A}_{1,1}$ with a subalgebra of the ring $\mathcal{R}[\lambda_1^{\pm 1}, \lambda_2^{\pm 1}].$

\begin{corollary}
The Laurent polynomial ring $\mathcal{R}[a^{\pm 1},c]$ is a subalgebra of $\mathcal{A}_{1,1}.$
\end{corollary}

\begin{remark} The algebra $\mathcal{A}_{1,1}$ actually admits a finite set of generators : $a,a^{-1},c,f.$ Indeed, in the next section we introduce elements $x^\star$ and $y^\star,$ and it can be seen that $f_{i,j}=(x^\star)^i(y^\star)^jf.$ Then in the proof of Proposition \ref{y star computation}, it is shown  $x^\star,y^\star$ are generated by $a,a^{-1},c,f.$ However, it is more complicated to describe the algebra using the finite list of generators.
\end{remark}

\subsection{Transparency and actions of \texorpdfstring{$\mathcal{A}$}{A} on \texorpdfstring{$\mathcal{A}_{1,1}$}{A1,1}}\label{transparency section}

We now describe two actions of $\mathcal{A}$ on $\mathcal{A}_{1,1}$ and use them to define transparent polynomials, analogous to the notion defined in \cite{Le15}.

\begin{definition}
We define an algebra homomorphism denoted by $(-)^{\star}: \mathcal{A} \rightarrow \mathcal{A}_{1,1}$ such that if $D \in \mathcal{A}$ is a diagram, the element $D^{\star} \in \mathcal{A}_{1,1}$ is the diagram of $D$ placed over the identity strand $1 \in \mathcal{A}_{1,1}.$ It can be checked diagrammatically that this is a well-defined algebra homomorphism.

Similarly, we define $(-)_\star: \mathcal{A} \rightarrow \mathcal{A}_{1,1}$ to be the algebra homomorphism such that if $D \in \mathcal{A}$ is a diagram, the element $D_{\star} \in \mathcal{A}_{1,1}$ is the diagram of $D$ placed below the identity strand $1 \in \mathcal{A}_{1,1}.$
\end{definition}

Recall that the algebra $\mathcal{A}$ is isomorphic to $\mathcal{R}[x,y].$ Below we show the images of the homomorphisms $(-)^\star$ and $(-)_\star$ on the generators $x,y$ of $\mathcal{A}.$

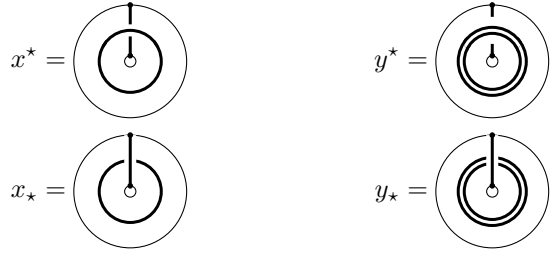
\begin{align*}
x^\star&=\xusebox{xUpStar} & y^\star&=\xusebox{yUpStar}\\
x_\star&=\xusebox{xDownStar} & y_\star&=\xusebox{yDownStar}
\end{align*}

The following definition uses the fact that $(-)^\star$ and $(-)_\star$ are algebra homomorphisms.

\begin{definition}\label{transparent definition}
A polynomial $S(x,y) \in \mathcal{A}$ is said to be \textit{transparent} if either of the following equivalent conditions hold:
\begin{enumerate}
    \item $S(x,y)^\star=S(x,y)_\star \in \mathcal{A}_{1,1}$
    \item $S(x^\star,y^\star)=S(x_\star,y_\star) \in \mathcal{A}_{1,1}.$
\end{enumerate}
\end{definition}

For example, the empty diagram in $\mathcal{A}$ is a transparent element regardless of the choice of $q \in \mathcal{R}.$ If $q=1$ then every element of $\mathcal{A}$ is transparent. When $q$ is chosen to be a root of unity, we will find more interesting elements which are transparent.

We now briefly explain the process which takes a transparent polynomial and constructs central elements in the skein algebra $\mathcal{S}_q^{G_2}(\Sigma)$ of a surface $\Sigma.$

\begin{definition}
Let a knot diagram $K$ on $\Sigma$ represent a framed knot in the thickened surface $\Sigma \times [0,1].$ Given a monomial $x^iy^j \in \mathcal{R}[x,y],$ the \textit{threading} of $x^iy^j$ along $K$ is denoted by $K^{[x^iy^j]}.$ It is obtained by taking $i+j$ parallel copies of $K$, in the direction of the framing, such that $i$ of these copies are colored by a single-strand and $j$ of these copies are colored by a double-strand. We can extend this definition linearly to obtain the threading $K^{[S(x,y)]}$ of a polynomial $S(x,y) \in \mathcal{R}[x,y]$ along $K.$
\end{definition}

\begin{lemma}
If an element $S(x,y) \in \mathcal{A}$ is transparent, then the threading $K^{[S]}$ of $S$ along any knot $K$ is a central element of $\mathcal{S}_q^{G_2}(\Sigma).$
\end{lemma}

\begin{proof}
The transparency property of $S(x,y)$ implies that the element $K^{[S]}$ can pass through any single-strand edge. To finish the proof of the lemma, it suffices to show that $K^{[S]}$ can be pass through a double-strand edge as well. This follows by using the second relation of (\ref{2gon}):
\begin{equation*}
 \xusebox{2--doubleVline} = -\frac{1}{[2]^2} \xusebox{doubleStrandDoublePop}, 
\end{equation*}

which shows that a small interval of a double-strand can be replaced by two single-strands.
\end{proof}

Threadings of polynomials can be defined similarly for links, and threadings of transparent polynomials along links yield central elements in the skein algebra as well.

\section{Power sum polynomials for \texorpdfstring{$G_2$}{G2}}

We now develop the ingredients to construct transparent polynomials.

\subsection{Polynomials in \texorpdfstring{$\mathcal{R}[\lambda_1^{\pm1},\lambda_2^{\pm1}]$}{R[l1,l2]} }

Consider the 2-variable Laurent polynomial algebra $E=\mathcal{R}[\lambda_1^{\pm1},\lambda_2^{\pm1}].$ Later, we will relate certain elements of this ring with elements in the algebra $\mathcal{A}_{1,1}.$

\begin{definition}\label{degree}
To a monomial $\lambda_1^i\lambda_2^j\in E$ we associate a degree $d_1$ and a bidegree $d_2$ given by 

\begin{align*}
d_1(\lambda_1^i\lambda_2^j)&=i+j \in \mathbb{Z}\\
d_2(\lambda_1^i\lambda_2^j)&=(i,j) \in \mathbb{Z}^2.\\
\end{align*}
\end{definition}

We note that $E$ has a natural basis consisting of monomials of distinct bidegree:

\begin{equation*}
\{\lambda_1^i\lambda_2^j \mid i,j \in \mathbb{Z}\}.
\end{equation*}

We can place a linear order on our monomials by using the lexicographic order and declaring $d_2(\lambda_1^i\lambda_2^j) > d_2(\lambda_1^{i'},\lambda_2^{j'})$ if either $i>i'$ or $i=i'$ and $j>j'.$

\begin{definition}\label{bidegree}
This linear order on monomials allows us to extend $d_1$ and $d_2$ to all of $E$ by the following. If $p=\sum a_{ij} \lambda_1^i\lambda_2^j$ is a finite $\mathcal{R}\text{-linear}$ combination of basis monomials, we define the (bi)degree of $p$ to be the (bi)degree of the top monomial of $p$:

\begin{align*}
d_1(p)&=\max \{d_1(\lambda_1^i\lambda_2^j) \mid a_{ij} \neq 0 \}\\
d_2(p)&=\max \{d_2(\lambda_1^i\lambda_2^j) \mid a_{ij} \neq 0 \}.   
\end{align*}

\end{definition}

Since $\mathcal{R}$ is an integral domain, the degrees satisfy the following property:

\begin{proposition}
If $p_1,p_2 \in E$, then

\begin{align*}
d_1(p_1p_2)&=d_1(p_1)+d_1(p_2)\\
d_2(p_1p_2)&=d_2(p_1)+d_2(p_2).    
\end{align*}
\end{proposition}

Consider the two elements $\lambda_1+\lambda_2$ and $\lambda_1\lambda_2 \in E.$ For $i \in \mathbb{Z}_{\geq 0}$ and $j \in \mathbb{Z}$ a monomial $(\lambda_1+\lambda_2)^i(\lambda_1\lambda_2)^j$ is a monic element in $E$ of bidegree $(i+j,j).$ Thus, these monomials are linearly independent over $\mathcal{R}$. Equivalently, $\lambda_1+\lambda_2$ and $\lambda_1\lambda_2$ are algebraically independent over $\mathcal{R}.$ We are interested in the subalgebra $E' \subseteq E$ generated by these variables:

\begin{equation*}
E'=\mathcal{R}[\lambda_1+\lambda_2,(\lambda_1\lambda_2)^{\pm1}].
\end{equation*}

Consequently, a basis of $E'$ is given by

\begin{equation*}
\{(\lambda_1+\lambda_2)^i(\lambda_1\lambda_2)^j \mid i \in \mathbb{Z}_{\geq 0} \text{ and } j \in \mathbb{Z} \}.
\end{equation*}

We are interested in two distinguished elements $\boldsymbol{x}, \boldsymbol{y} \in E'$ given by

\begin{align*}
\boldsymbol{x}=&\lambda_1+\lambda_2+\lambda_1\lambda_2+1+(\lambda_1\lambda_2)^{-1}+\lambda_2^{-1}+\lambda_1^{-1} \\
\boldsymbol{y}=& \lambda_1^2\lambda_2+\lambda_1\lambda_2^2+\lambda_1\lambda_2+\lambda_1+\lambda_2+\lambda_1\lambda_2^{-1}+1 \\
& +1+\lambda_1^{-1}\lambda_2+\lambda_2^{-1}+\lambda_1^{-1}+(\lambda_1\lambda_2)^{-1}+(\lambda_1\lambda_2^{2})^{-1}+(\lambda_1^2\lambda_2)^{-1}.\\
\end{align*}

Using $(\lambda_1\lambda_2)^{-1}(\lambda_1+\lambda_2)=\lambda_1^{-1}+\lambda_2^{-1},$ it is an easy exercise to show that $\boldsymbol{x},\boldsymbol{y} \in E'.$ The expression defining the element $\boldsymbol{x}$ models the trace of a matrix in the 7-dimensional fundamental representation of $G_2$ while $\boldsymbol{y}$ models the same for the 14-dimensional fundamental representation.

We have written $\boldsymbol{x}$ as a sum of 7 terms and $\boldsymbol{y}$ as a sum of 14 terms. We remark that $\boldsymbol{y}=e_2-\boldsymbol{x},$ where $e_2$ is the second elementary symmetric sum (see Definition \ref{elementary sums definition} below) of the terms of $\boldsymbol{x}.$

\begin{proposition}
The elements $\boldsymbol{x}$ and $\boldsymbol{y}$ are algebraically independent over $\mathcal{R}.$
\end{proposition}

\begin{proof}
We view each monomial $\boldsymbol{x}^i\boldsymbol{y}^j$ as a monic element of $E=\mathcal{R}[\lambda_1^{\pm1},\lambda_2^{\pm1}]$ and compute its bidegree $d_2(\boldsymbol{x}^i\boldsymbol{y}^j)=d_2(\boldsymbol{x}^i)+d_2(\boldsymbol{y}^j)=i\cdot d_2(\boldsymbol{x})+j \cdot d_2(\boldsymbol{y})=(i,i)+(2j,j)=(i+2j,i+j).$ We see the set of monomials $\{\boldsymbol{x}^i\boldsymbol{y}^j\}_{i,j \geq 0}$ consists of monic elements of $E$ of distinct bidegrees. Therefore, the monomials are linearly independent over $\mathcal{R}$ and $\boldsymbol{x},\boldsymbol{y}$ are algebraically independent.
\end{proof}

More generally, we will be using the elements of $E':$

\begin{align*}
\boldsymbol{x}^{(i)}=&\lambda_1^i+\lambda_2^i+(\lambda_1\lambda_2)^i+1+(\lambda_1\lambda_2)^{-i}+\lambda_2^{-i}+\lambda_1^{-i} \\
\boldsymbol{y}^{(i)}=& (\lambda_1^2\lambda_2)^i+(\lambda_1\lambda_2^2)^i+(\lambda_1\lambda_2)^i+\lambda_1^i+\lambda_2^i+(\lambda_1\lambda_2^{-1})^i+1 \\
& +1+(\lambda_1^{-1}\lambda_2)^i+\lambda_2^{-i}+\lambda_1^{-i}+(\lambda_1\lambda_2)^{-i}+(\lambda_1\lambda_2^{2})^{-i}+(\lambda_1^2\lambda_2)^{-i}.\\
\end{align*}

Using this notation, $\boldsymbol{x}^{(1)}=\boldsymbol{x}$ and $\boldsymbol{y}^{(1)}=\boldsymbol{y}.$ We remark that $\boldsymbol{x}^{(i)}$ gives the $i$th power sum (Definition \ref{elementary sums definition}) of the summands of $\boldsymbol{x},$ and similarly for $\boldsymbol{y}.$

\subsection{Polynomials in \texorpdfstring{$\mathcal{R}[x,y]$}{R[x,y]}}\label{Pn,Qn}\label{def polynomials}

Consider the polynomial ring in two variables $\mathcal{R}[x,y].$ We will now study some polynomials in this ring from a purely algebraic perspective before we return to studying the diagrammatically defined algebra $\mathcal{A},$ which was shown in Proposition \ref{A isomorphic} to be isomorphic to $\mathcal{R}[x,y].$

Define the polynomials $P_k(x,y)$ and $Q_k(x,y)$ in $\mathcal{R}[x,y]$ by the following recursive relations.

\begin{equation*}
P_k(x,y)=
\begin{cases}
7 & k=0\\
\sum_{i=1}^{k-1} (-1)^{i-1}e_iP_{k-i}(x,y)+(-1)^{k-1}ke_k & 0<k< 7\\
\sum_{i=1}^7 (-1)^{i-1}e_iP_{k-i}(x,y) & k\geq 7\\
\end{cases}
\end{equation*}

\begin{equation*}
Q_k(x,y)=
\begin{cases}
14 & k=0\\
\sum_{i=1}^{k-1} (-1)^{i-1}f_iQ_{k-i}(x,y)+(-1)^{k-1}kf_k & 0<k< 14\\
\sum_{i=1}^{14} (-1)^{i-1}f_iQ_{k-i}(x,y) & k\geq 14,\\
\end{cases}
\end{equation*}

where the coefficients $e_i:=e_i(x,y)$ and $f_i:=f_i(x,y)$ are given by the following:

\begin{align*}
e_0&=1 & f_0&=1\\
e_1&=x & f_1&=y\\
e_2&=x+y& f_2&=x^3 - x^2 - 2xy - x\\
e_3&=x^2-y & f_3&=x^4 - x^3 - 3x^2y - x^2 + 2y^2 + x + y\\
e_k&=e_{7-k} \text{ if } 0 \leq k \leq 7& f_4&=x^3y - x^3 - x^2y - 2xy^2 + x^2 + xy - y^2 + x + y\\
\end{align*}
 and
\begin{align*}
f_5&=x^5 - 2x^4 - 5x^3y + 3x^2y + 6xy^2 + y^3 + 2x^2 + 5xy + 2y^2 - x \\
f_6&=x^4 - 3x^3y + x^2y^2 - x^2y + 4xy^2 - 2x^2 + 3xy + 2y^2 + y \\
f_7&=-2x^5 + 4x^4 + 6x^3y + 2x^2y^2 + 2x^3 - 4x^2y - 8xy^2 - 2y^3 - 6x^2 - 6xy - 6y^2 + 2 \\
f_k&=f_{14-k}  \text{ if } 0 \leq k \leq 14\\
\end{align*}

The first few terms of each sequence of polynomials are as follows:

\begin{align*}
P_0(x,y)&=7 & Q_0(x,y)&=14\\
P_1(x,y)&=x & Q_1(x,y) &=y\\
P_2(x,y)&=x^2-2x-2y &Q_2(x,y)&=y^2-2x^3+2x^2+4xy+2x.\\
\end{align*}

\begin{definition}\label{elementary sums definition}(Elementary symmetric sums and power sums)
Given terms $t_1,...,t_s,$ the $i\text{-th}$ \textit{elementary symmetric sum} of the terms is defined to be

\begin{equation*}
\sum_{1\leq j_1 < j_2 < \cdots< j_i\leq s} t_{j_1}t_{j_2}\cdots t_{j_i},
\end{equation*}

while the $i\text{-th}$ \textit{power sum} of the terms is defined to be

\begin{equation*}
\sum_{j=1}^s t_j^i.
\end{equation*}

\end{definition}

The following proposition relates the coefficients defining $P_n$ and $Q_n$ to elementary symmetric sums of the summands of $\boldsymbol{x}$ and $\boldsymbol{y}.$ It can be verified by direct computation, which is made easier by use of mathematical software.

\begin{proposition}\label{elementary sums}
For each $0 \leq i \leq 7,$ the quantity $e_i(\boldsymbol{x}^{(1)},\boldsymbol{y}^{(1)})$ is equal to the $i\text{-th}$ elementary symmetric sum of the terms of $\boldsymbol{x}^{(1)}.$ Similarly, for each $0 \leq i \leq 14,$ the quantity $f_i(\boldsymbol{x}^{(1)},\boldsymbol{y}^{(1)})$ is equal to the $i\text{-th}$ elementary symmetric sum of the terms of $\boldsymbol{y}^{(1)}.$ 
\end{proposition}

We can now relate our polynomials $P_k$ and $Q_k$ to the $k\text{-th}$ power sums of the terms of $\boldsymbol{x}$ and $\boldsymbol{y}.$

\begin{proposition}
For each $k \geq 0,$ the following holds:
\begin{align*}
P_k(\boldsymbol{x}^{(1)},\boldsymbol{y}^{(1)})&=\boldsymbol{x}^{(k)}\\
Q_k(\boldsymbol{x}^{(1)},\boldsymbol{y}^{(1)})&=\boldsymbol{y}^{(k)}.\\
\end{align*}
More generally, for each $k \geq 0$ and for each $i \geq 0,$

\begin{align*}
P_k(\boldsymbol{x}^{(i)},\boldsymbol{y}^{(i)})&=\boldsymbol{x}^{(ik)}\\
Q_k(\boldsymbol{x}^{(i)},\boldsymbol{y}^{(i)})&=\boldsymbol{y}^{(ik)}.\\
\end{align*}

\end{proposition}

\begin{proof}
We consider $P_k.$ The case of $Q_k$ is handled similarly.

By Proposition \ref{elementary sums}, the quantity $e_l(\boldsymbol{x}^{(1)},\boldsymbol{y}^{(1)})$ is equal to the $l\text{-th}$ elementary symmetric sum of the terms of $\boldsymbol{x}^{(1)}.$ Let $A$ be a 7-by-7 diagonal matrix such that each diagonal entry $A_{jj}$ is the $j\text{-th}$ summand of $\boldsymbol{x}^{(1)}.$ Each $e_l$ is a coefficient of the characteristic polynomial of $A.$ The Cayley-Hamilton theorem applied to the matrix $A^k$ then gives

\begin{align*}
P_k(\boldsymbol{x}^{(1)},\boldsymbol{y}^{(1)})&=\text{Tr}A^k\\
&=\boldsymbol{x}^{(k)}.
\end{align*}

Next, consider the algebra endomorphism $\phi_i:E \rightarrow E$ defined on the generators by $\phi_i(\lambda_1)=\lambda_1^i$ and $\phi_i(\lambda_2)=\lambda_2^i.$ Then $\phi_i(\boldsymbol{x}^{(1)})=\boldsymbol{x}^{(i)}$ and $\phi_i(\boldsymbol{y}^{(1)})=\boldsymbol{y}^{(i)}$.

We then use $\phi_i$ and the first part of the proposition to see that

\begin{align*}
P_k(\boldsymbol{x}^{(i)},\boldsymbol{y}^{(i)})&=P_k(\phi_i(\boldsymbol{x}^{(1)}),\phi_i(\boldsymbol{y}^{(1)}))\\
&=\phi_i(P_k(\boldsymbol{x}^{(1)},\boldsymbol{y}^{(1)}))\\
&=\phi_i(\boldsymbol{x}^{(k)})\\
&=\boldsymbol{x}^{(ik)}.\\
\end{align*}

\end{proof}

\begin{corollary}\label{power sum corollary}
For each $k \geq 0$ and for each $i \geq 0$ the following holds in $\mathcal{R}[x,y]:$

\begin{align*}
P_k(P_i(x,y),Q_i(x,y))&=P_{ik}(x,y)\\
Q_k(P_i(x,y),Q_i(x,y))&=Q_{ik}(x,y).\\
\end{align*}
\end{corollary}

\begin{proof}
Since $\boldsymbol{x}^{(1)}$ and $\boldsymbol{y}^{(1)}$ are algebraically independent, there is an injective algebra homomorphism $\psi: \mathcal{R}[x,y] \rightarrow E$ defined on generators by $x \mapsto \boldsymbol{x}^{(1)}$ and $y \mapsto \boldsymbol{y}^{(1)}.$
Applying $\psi$ to the left side of the first equation in the proposition yields

\begin{align*}
\psi(P_k(P_i(x,y),Q_i(x,y)))&=P_k(P_i(\boldsymbol{x}^{(1)},\boldsymbol{y}^{(1)}),Q_i(\boldsymbol{x}^{(1)},\boldsymbol{y}^{(1)}))\\
&=P_k(\boldsymbol{x}^{(i)},\boldsymbol{y}^{(i)})\\
&=P_{ki}(\boldsymbol{x}^{(1)},\boldsymbol{y}^{(1)}).\\
\end{align*}
Applying $\psi$ to the right side of the first equation in the proposition yields the same answer. The injectivity of $\psi$ then proves the proposition for $P_k.$ The case of $Q_k$ is proven similarly.
\end{proof}

\begin{definition}
For a monomial $x^iy^j \in \mathcal{R}[x,y],$ we define its bidegree to be
\begin{equation*}
D_2(x^iy^j)=(i+2j,i+j) \in \mathbb{Z}_{\geq 0}.
\end{equation*}

Using a lexicographic ordering of bidegrees, we extend $D_2$ to $\mathcal{R}[x,y]$ by defining for an element $p=\sum a_{ij}x^iy^j \in \mathcal{R}[x,y]$ its bidegree

\begin{equation*}
D_2(p)=\max \{D_2(x^iy^j) \mid a_{ij} \neq 0\}.
\end{equation*}
\end{definition}

The bidegree satisfies the following compatibility with multiplication in $\mathcal{R}[x,y].$

\begin{proposition}
For any two polynomials $S(x,y),T(x,y) \in \mathcal{R}(x,y)$ we have
\begin{equation*}
D_2(S(x,y)T(x,y))=D_2(S(x,y))+D_2(T(x,y)).
\end{equation*}
\end{proposition}

The following propositions provide support for the unusual choice of the bidegree $D_2.$

\begin{proposition}
For $k \geq 1,$ the polynomials $P_k(x,y)$ and $Q_k(x,y)$ are monic polynomials in $\mathcal{R}[x,y].$ Their bidegrees are $D_2(P_k)=(k,k)$ and $D_2(Q_k)=(2k,k).$
\end{proposition}

\begin{proof}
By inspecting the coefficients $e_i$ and $f_i$ in the recursive definitions of $P_k$ and $Q_k$ we observe inductively that $P_k=xP_{k-1}+\text{ (lower terms) }$ and $Q_k=yQ_{k-1}+ \text{ (lower terms)}$. Thus, $D_2(P_k)=D_2(x)+D_2(P_{k-1})=(1,1)+D_2(P_{k-1})$ and $D_2(Q_k)=D_2(y)+D_2(Q_{k-1})=(2,1)+D_2(Q_{k-1}).$ The result follows by induction.
\end{proof}

Thus, $P_0$ and $Q_0$ are the only non-monic polynomials in our sequences of polynomials. In the following proposition we renormalize $P_0$ and $Q_0$ so that $P_0(x,y)=1=Q_0(x,y)$ (but the other polynomials stay unchanged).

\begin{proposition}\label{P,Q basis}
The set of products of polynomials
\begin{equation*}
\{P_k(x,y)Q_l(x,y)\}_{k,l\geq 0}
\end{equation*}
forms a basis for the 2-variable polynomial ring $\mathcal{R}[x,y].$
\end{proposition}

\begin{proof}
We observe that each $P_kQ_l$ is a monic polynomial of $\mathcal{R}[x,y]$ of distinct bidegree $D_2(P_kQ_l)=D_2(P_k)+D_2(Q_l)=(k,k)+(2l,l)=(k+2l,k+l).$ The standard basis $\{x^ky^l\}$ of $\mathcal{R}[x,y]$ also consists of monic polynomials of distinct bidegrees $D_2(x^ky^l)=(k+2l,k+l).$ Thus, there is a unitriangular change of basis from the standard basis $\{x^ky^l\}_{k,l \geq 0}$ to $\{P_k(x,y)Q_l(x,y)\}_{k,l\geq 0}.$
\end{proof}

\begin{proposition}
For $i,j \geq 0$ we have

\begin{equation*}
D_2(x^iy^j)=d_2(\boldsymbol{x}^{(i)}\boldsymbol{y}^{(j)}).
\end{equation*}
\end{proposition}

\begin{proof}
This follows from the compatibility of $D_2$ and $d_2$ with multiplication and the following observation:

\begin{align*}
D_2(x^i)&=(i,i)=d_2(\boldsymbol{x}^{(i)})\\
D_2(y^j)&=(2j,j)=d_2(\boldsymbol{y}^{(j)}).
\end{align*}
\end{proof}

\section{Transparent elements}

In this section we connect the algebra $E'$ to the algebra $\mathcal{A}_{1,1}$ and use the relationship to show that the polynomials $P_n$ and $Q_n$ are transparent when $q^{2n}=1.$

\subsection{The maps \texorpdfstring{$F^\star$}{F*} and \texorpdfstring{$F_\star$}{F*}}

We next define the maps $F^\star$ and $F_\star,$ which are the key ingredients in connecting our algebra $E'=\mathcal{R}[\lambda_1+\lambda_2,(\lambda_1\lambda_2)^{\pm 1}]$ and our diagrammatic algebra $\mathcal{A}_{1,1}.$

\begin{definition} \label{F star}
Below we define algebra maps $F^\star, F_\star: E' \rightarrow \mathcal{A}_{1,1}$ on generators.

\begin{align*}
F^\star(\lambda_1\lambda_2)&=q^2a & F_\star(\lambda_1\lambda_2)&=q^{-2}a\\
F^\star(\lambda_1+\lambda_2)&=\frac{q}{[2]}(c-a-1)& F_\star(\lambda_1+\lambda_2)&=\frac{q^{-1}}{[2]}(c-a-1)\\
\end{align*}
\end{definition}

In particular, the map $F_\star$ is obtained from the map $F^\star$ by replacing $q$ with $q^{-1}.$ The following proposition provides the precise relationship between our distinguished elements of $E'$ and our distinguished elements of $\mathcal{A}_{1,1}.$

\begin{proposition}\label{y star computation}
$F^\star$ and $F_\star: E' \rightarrow \mathcal{A}_{1,1}$ are well-defined algebra homomorphisms and satisfy the following
\begin{align*}
F^\star(\boldsymbol{x}^{(1)})&=x^\star & F^\star(\boldsymbol{y}^{(1)})&=\bar{y}:=y^\star + \frac{1}{[2]^2} f\\
F_\star(\boldsymbol{x}^{(1)})&=x_\star & F^\star(\boldsymbol{y}^{(1)})&=\underline{y}:=y_\star +\frac{1}{[2]^2}f.\\
\end{align*}
\end{proposition}

\begin{proof}
We will prove the statement for $F^\star.$ The statement for $F_\star$ is proved by replacing $q$ with $q^{-1}.$ As $\lambda_1+\lambda_2$ and $(\lambda_1\lambda_2)^{\pm 1}$ generate $E'=\mathcal{R}[\lambda_1+\lambda_2,(\lambda_1\lambda_2)^{\pm 1}],$ we see that $F^\star$ is a well-defined algebra homomorphism by noting that $q^2a$ is an invertible element of $\mathcal{A}_{1,1}$ and that $\mathcal{A}_{1,1}$ is a commutative algebra.

To check that $F^\star(\boldsymbol{x}^{(1)})=x^\star,$ we compute diagrammatically that

\begin{align*}
x^\star&=\xusebox{xUpStar}\\
&=\frac{1}{[2]}\left( q^3\xusebox{A11--a}+q^{-3}\xusebox{A11--aInv}+q\xusebox{A11--c}+q^{-1}\xusebox{A11--aInvc}\right)\\
&=\frac{1}{[2]}(q^3a+q^{-3}a^{-1}+qc+q^{-1}a^{-1}c).\\
\end{align*}

On the other hand, we compute that

\begin{align*}
F^\star(\boldsymbol{x}^{(1)})=& F^\star(\lambda_1+\lambda_2)+F^\star(\lambda_1\lambda_2)+F^\star((\lambda_1\lambda_2)^{-1})+F^\star(1)\\
&+F^\star((\lambda_1\lambda_2)^{-1}(\lambda_1+\lambda_2))\\
=&\frac{q}{[2]}(c-a-1)+q^2a+q^{-2}a^{-1}+1+q^{-2}a^{-1}\frac{q}{[2]}(c-a-1)\\
=&\frac{1}{[2]}(q^3a+q^{-3}a^{-1}+qc+q^{-1}a^{-1}c).\\
\end{align*}

Therefore, $F^\star(\boldsymbol{x}^{(1)})=x^\star.$

Now, to check that $F^\star(\boldsymbol{y}^{(1)})=\bar{y}=y^\star + \frac{1}{[2]^2} f,$ we compute diagrammatically that

\begin{align}
y^\star=&\xusebox{yUpStar} \nonumber \\
=&\frac{1}{[2]}\left( q^3\xusebox{A11--yc} +q^{-3}\xusebox{A11--yInvc}+\frac{1}{[2]}\xusebox{A11--yBox}\right). \label{y star}
\end{align}

Next, we apply the double-edge removal relation to all three diagrams and then continue to fully simplify. The following is the calculation for the third diagram of (\ref{y star}).

\begin{align*}
\xusebox{A11--yBox}&=-\xusebox{A11--f}+\frac{[4][6]}{[2][12]} \xusebox{A11--2gon}+\xusebox{A11--aInvc2}+\frac{1}{[3]}\xusebox{A11--xBox}\\
&=-f-\frac{[6][3][8]}{[2][12]}1+a^{-1}c^2+\frac{1}{[3]}([3](a+a^{-1})-\frac{[4]}{[2]}(c+a^{-1}c)).
\end{align*}

The other two diagrams of (\ref{y star}) are reduced using the same process and we record the results:

\begin{align*}
\xusebox{A11--yc}&=-a^2+\frac{[4][6]}{[2][12]}1+ac+\frac{1}{[3]}c\\
\xusebox{A11--yInvc}&=-a^{-2}+\frac{[4][6]}{[2][12]}1+a^{-2}c+\frac{1}{[3]}a^{-1}c.
\end{align*}

Combining these computations and substituting them into Equation (\ref{y star}), we write $y^\star$ in our basis for $\mathcal{A}_{1,1}$ as

\begin{align*}
y^\star=&\frac{1}{[2]} \left( -q^3 a^2 - q^{-3} a^{-2} + q^3 ac + q^{-3} a^{-2} c + \frac{1}{[2]} a^{-1} c^2 \right.\\
    &\left. + \frac{1}{[2]} a + \frac{1}{[2]} a^{-1} + \frac{q^2-1}{[2]} c + \frac{q^{-2}-1}{[2]} a^{-1} c - \frac{q^2+q^{-2}}{[2]} 1-\frac{1}{[2]}f \right).
\end{align*}

On the other hand, we compute

\begin{align*}
F^\star(\boldsymbol{y}^{(1)})=&F^*(\lambda_1 + \lambda_2) + F^*(\lambda_1^{-1} + \lambda_2^{-1}) + F^*(\lambda_1\lambda_2) + F^*(\lambda_1^{-1} \lambda_2^{-1}) + F^*(\lambda_1\lambda_2 (\lambda_1 + \lambda_2)) \\
&+ F^*(\lambda_1^{-1}\lambda_2^{-1} (\lambda_1^{-1} + \lambda_2^{-1})) + F^*((\lambda_1 + \lambda_2)(\lambda_1^{-1} + \lambda_2^{-1}))\\
=& q\frac{c-a-1}{[2]} + q^{-1}\frac{a^{-1}(c-a-1)}{[2]} + q^2 a + q^{-2} a^{-1} + q^2 a \left( q \frac{c-a-1}{[2]} \right) \\
&+ q^{-2}a^{-1}\left( q^{-1} \frac{a^{-1}(c-a-1)}{[2]} \right) + \left( q \frac{c-a-1}{[2]} \right) \left( q^{-1} \frac{a^{-1}(c-a-1)}{[2]} \right)\\
=&\frac{1}{[2]} \left( -q^3 a^2 - q^{-3} a^{-2} + q^3 ac + q^{-3} a^{-2} c + \frac{1}{[2]} a^{-1} c^2 \right.\\
&\left. + \frac{1}{[2]} a + \frac{1}{[2]} a^{-1} + \frac{q^2-1}{[2]} c + \frac{q^{-2}-1}{[2]} a^{-1} c - \frac{q^2+q^{-2}}{[2]} 1 \right).
\end{align*}

Thus, $F^\star(\boldsymbol{y}^{(1)})=y^\star+\frac{1}{[2]^2}f,$ as claimed.
\end{proof}

\begin{proposition}
The maps $F^\star$ and $F_\star$ each give an isomorphism $E' \cong \mathcal{R}[a^{\pm 1}, c] \subseteq \mathcal{A}_{1,1}.$ In particular, $F^\star$ and $F_\star$ are injective.    
\end{proposition}

\begin{proof}
We check the statement for $F^\star$. The other case is similar.
By recalling the Definition \ref{F star} of $F^\star$ on generators of $E',$ we see that the image of $F^\star$ is contained in the subalgebra $\mathcal{R}[a^{\pm1},c] \subseteq \mathcal{A}_{1,1}.$ Since $[2]$ is invertible in $\mathcal{R}$ we deduce from the definition of $F^\star$ on the generators of $E'$ that $a^{\pm 1},c$ are in the image of $F^\star.$ Thus, $F^\star$ is surjective.

 On the subalgebra $\mathcal{R}[a^{\pm 1}, c],$ for $k \in \mathbb{Z}$ and $j \in \mathbb{Z}_{\geq 0},$ we can assign the bidegree $(j,k)$  to the monomial $a^{k}c^j$ and give the set of all bidegrees the lexicographic ordering. We then see that $F^\star$ sends the monomial $(\lambda_1\lambda_2)^i(\lambda_1+\lambda_2)^j$ to an element of $\mathcal{R}[a^{\pm1},c]$ whose top bidegree is $(j,i+j).$ Thus, $F^\star$ sends a spanning set of $E'$ to a linearly independent subset of $\mathcal{A}_{1,1}$. Therefore, $F^\star$ is injective.  
\end{proof}

\subsection{The error term}\label{error term}

Our original strategy was to attempt to define an isomorphism from a subalgebra of $E=\mathcal{R}[\lambda_1^{\pm1},\lambda_2^{\pm1}]$ to $\mathcal{A}_{1,1}$ whose image on $\boldsymbol{x}^{(1)}$ was $x^\star$ and whose image on $\boldsymbol{y}^{(1)}$ was $y^\star.$ We did not find it possible to achieve this goal on the nose. Firstly, the algebra $\mathcal{A}_{1,1}$ contains zero divisors while $E$ does not. Indeed, as observed earlier, we have the identity $af=a$ and so $(a-1)f=0$ while both $a-1$ and $f$ are nonzero by consideration of our basis for $\mathcal{A}_{1,1}.$ 

Instead, we found an injective algebra homomomorphism $F^\star$ from the subalgebra $E' \subset E$ onto the subalgebra $\mathcal{R}[a^{\pm 1}, c] \subset \mathcal{A}_{1,1}.$ We did succeed in defining $F^\star$ so that its image on $\boldsymbol{x}^{(1)}$ is $x^\star.$ However, its image on $\boldsymbol{y}^{(1)}$ is not quite $y^\star.$ It is off by an error term: $F^\star(\boldsymbol{y}^{(1)})=y^\star+\frac{1}{[2]^2}f.$ Below we show that the error term is essentially negligible in our context since the diagram of $f$ allows loops to pass through it.

For example, we observe the following identity in $\mathcal{A}_{1,1}: y^\star f= y_\star f.$

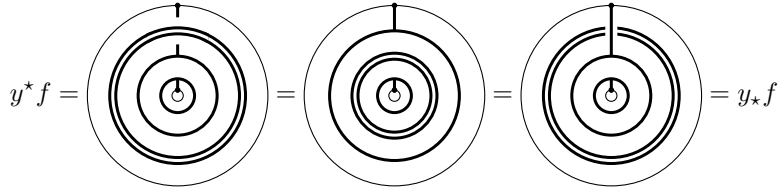
\begin{align*}
y^\star f = \xusebox{yUpStarf}=\xusebox{yf}=\xusebox{yDownStarf}=y_\star f
\end{align*}

More generally, we will make use of the following observation.
\begin{lemma}\label{f transparency}
The following identities hold in $\mathcal{A}_{1,1}:$
\begin{align*}
x^\star f&=x_\star f\\
y^\star f&=y_\star f\\
\overline{y}f&=\underline{y}f.\\
\end{align*}
\end{lemma}

\begin{proof}
The first two identities follow diagrammatically by using a 3-dimensional isotopy to pass a loop through the empty annular region in the diagram of $f.$ The third identity follows from the second and from the defining equations $\overline{y}:=y^\star+\frac{1}{[2]^2}f$ and $\underline{y}:=y_\star + \frac{1}{[2]^2}f.$
\end{proof}

\begin{proposition}\label{transparent alternative}
An element $S(x,y) \in \mathcal{A}$ is transparent (Definition \ref{transparent definition}) if and only if
\begin{equation*}
S(x^\star, \overline{y})=S(x_\star, \underline{y}) \in \mathcal{A}_{1,1}.    
\end{equation*}
\end{proposition}

\begin{proof}
It will suffice to show that
\begin{equation*}
S(x^\star,y^\star)-S(x_\star, y_\star)=S(x^\star,\overline{y})-S(x_\star,\underline{y}).
\end{equation*}
Write $S(x,y) \in \mathcal{A}$ as $S(x,y)=\sum a_{ij}x^iy^j$ for $a_{ij} \in \mathcal{R}.$

Then
\begin{align*}
S(x^\star, y^\star)&=\sum a_{ij}(x^\star)^i(y^\star)^j\\
&=\sum a_{ij}(x^\star)^i(\overline{y}-\frac{1}{[2]^2}f)^j\\
&=\sum a_{ij}(x^\star)^i(\sum_{k=0}^j \binom{j}{k}(\overline{y})^k(-\frac{1}{[2]^2}f)^{j-k})\\
&=\sum a_{ij}(x^\star)^i(\overline{y})^j+\sum_{k<j} a_{ij}\binom{j}{k}(x^\star)^i(\overline{y})^k(-\frac{1}{[2]^2}f)^{j-k}\\
&=S(x^\star,\overline{y})+\sum_{k<j} a_{ij}\binom{j}{k}(x_\star)^i(\underline{y})^k(-\frac{1}{[2]^2}f)^{j-k}.\\
\end{align*}

The last equality follows from Lemma \ref{f transparency}. Similarly, we compute that

\begin{equation*}
S(x_\star, y_\star)=S(x_\star,\underline{y})+\sum_{k<j} a_{ij} \binom{j}{k}(x_\star)^i(\underline{y})^k(-\frac{1}{[2]^2}f)^{j-k},
\end{equation*}
and the result follows.
\end{proof}

\subsection{The transparent elements}

Recall the definition of the degree $d_1$ from Definition \ref{degree}.

\begin{proposition}
Suppose $\lambda \in E' \subseteq E=\mathcal{R}[\lambda_1^{\pm1},\lambda_2^{\pm1}]$ is a homogeneous element of degree $d_1(\lambda)=k.$ Then $F^\star(\lambda)=q^{2k}F_\star(\lambda).$
\end{proposition}

\begin{proof}
Since the generators $\lambda_1 +\lambda_2$ and $(\lambda_1\lambda_2)^{\pm1}$ are homogenous elements of degree $1$ and degree $\pm2$, respectively, it suffices to check the statement on these generators. We observe that $F^\star(\lambda_1\lambda_2)=q^2a=q^4(q^{-2}a)=q^4F_\star(\lambda_1\lambda_2).$ Similarly, $F^\star(\lambda_1+\lambda_2)=\frac{q}{[2]}(c-a-1)=q^2\frac{q^{-1}}{[2]}(c-a-1)=q^2F_\star(\lambda_1\lambda_2).$
\end{proof}

Recall the polynomials $P_n$ and $Q_n$ and their power sum properties from Section \ref{def polynomials}. We are now able to show that $P_n$ and $Q_n$ are transparent when $q$ is a $2n\text{-th}$ root of unity.

\begin{theorem}\label{transparent}
If $q \in \mathcal{R}$ satisfies $q^{2n}=1,$ then both the elements $P_n(x,y)$ and $Q_n(x,y)$ in $\mathcal{A}$ are transparent elements.
\end{theorem}

\begin{proof}

We compute

\begin{align*}
P_n(x^\star,\overline{y})&=P_n(F^\star(\boldsymbol{x}^{(1)}),F^\star(\boldsymbol{y}^{(1)}))\\
&=F^\star(P_n(\boldsymbol{x}^{(1)},\boldsymbol{y}^{(1)}))\\
&=F^\star(\boldsymbol{x}^{(n)})\\
&=F^\star(\lambda_1^n+\lambda_2^n)+F^\star(\lambda_1^{-n}+\lambda_2^{-n})+F^\star((\lambda_1\lambda_2)^n)+F^\star((\lambda_1\lambda_2)^{-n})+F^\star(1)\\
&=q^{2n}F_\star(\lambda_1^n+\lambda_2^n)+q^{-2n}F_\star(\lambda_1^{-n}+\lambda_2^{-n})+q^{4n}F_\star((\lambda_1\lambda_2)^n)+q^{-4n}F_\star((\lambda_1\lambda_2)^{-n})+F_\star(1).\\
\end{align*}

Similarly, we compute

\begin{equation*}
P_n(x_\star,\underline{y})=F_\star(\lambda_1^n+\lambda_2^n)+F_\star(\lambda_1^{-n}+\lambda_2^{-n})+F_\star((\lambda_1\lambda_2)^n)+F_\star((\lambda_1\lambda_2)^{-n})+F_\star(1).\\
\end{equation*}

Thus, when $q^{2n}=1$ we have $P_n(x^\star,\overline{y})=P_n(x_\star,\underline{y}),$ and $P_n$ is transparent by Proposition \ref{transparent alternative}. For the case of $Q_n$ we perform a similar computation

\begin{align*}
Q_n(x^\star, \overline{y})=& F^\star(\boldsymbol{y}^{(n)})\\
=& F^\star((\lambda_1^2\lambda_2)^n+(\lambda_1\lambda_2^2)^n)+F^\star((\lambda_1\lambda_2)^n)+F^\star(\lambda_1^n+\lambda_2^n)+F^\star(\lambda_1\lambda_2^{-1})^n+(\lambda_1^{-1}\lambda_2)^n) \\
& +F^\star(1)+F^\star(1)+F^\star(\lambda_2^{-n}+\lambda_1^{-n})+F^\star((\lambda_1\lambda_2)^{-n})+F^\star((\lambda_1\lambda_2^{2})^{-n}+(\lambda_1^2\lambda_2)^{-n}).\\
=&q^{6n}F_\star((\lambda_1^2\lambda_2)^n+(\lambda_1\lambda_2^2)^n)+q^{4n}F_\star((\lambda_1\lambda_2)^n)+q^{2n}F_\star(\lambda_1^n+\lambda_2^n)+F_\star((\lambda_1\lambda_2^{-1})^n+(\lambda_1^{-1}\lambda_2)^n) \\
& +F_\star(1)+F_\star(1)+q^{-2n}F_\star(\lambda_2^{-n}+\lambda_1^{-n})+q^{-4n}F_\star((\lambda_1\lambda_2)^{-n})+q^{-6n}F_\star((\lambda_1\lambda_2^{2})^{-n}+(\lambda_1^2\lambda_2)^{-n}).\\
\end{align*}

to see that $Q_n$ is transparent when $q^{2n}=1.$
\end{proof}

\begin{notation}Inspired by the computations in the previous proof, and to make the proof of the upcoming Theorem \ref{uniqueness} cleaner, we will use the following notation.

\begin{align*}
\tilde{\boldsymbol{x}}^{(i)}=&q^{2i}(\lambda_1^i+\lambda_2^i)+q^{-2i}(\lambda_1^{-i}+\lambda_2^{-i})+q^{4i}(\lambda_1\lambda_2)^i+q^{-4i}(\lambda_1\lambda_2)^{-i}+1\\
\tilde{\boldsymbol{y}}^{(j)}=&q^{6j}((\lambda_1^2\lambda_2)^j+(\lambda_1\lambda_2^2)^j)+q^{4j}(\lambda_1\lambda_2)^j+q^{2j}(\lambda_1^j+\lambda_2^j)+(\lambda_1\lambda_2^{-1})^j+(\lambda_1^{-1}\lambda_2)^j \\
& +1+1+q^{-2j}(\lambda_2^{-j}+\lambda_1^{-j})+q^{-4j}(\lambda_1\lambda_2)^{-j}+q^{-6j}((\lambda_1\lambda_2^{2})^{-j}+(\lambda_1^2\lambda_2)^{-j}).\\
\end{align*}

\end{notation}

By the injectivity of $F_\star,$ the elements $\tilde{\boldsymbol{x}}^{(i)},\tilde{\boldsymbol{y}}^{(j)}  \in E'$ are the unique elements satisfying 

\begin{align*}
F^\star(\boldsymbol{x}^{(i)})&=F_\star(\tilde{\boldsymbol{x}}^{(i)})\\
F^\star(\boldsymbol{y}^{(j)})&=F_\star(\tilde{\boldsymbol{y}}^{(j)}).
\end{align*}

We now discuss some terms of $\boldsymbol{x}^{(i)}\boldsymbol{y}^{(j)}$ which behave like leading terms. Recall the bidegree $d_2$ from Definition \ref{bidegree}. The top two terms of $\boldsymbol{x}^{(i)}\boldsymbol{y}^{(j)}$ , with respect to the bidegree $d_2$ are $\lambda_1^{i+2j}\lambda_2^{i+j}$ and $\lambda_1^{i+2j}\lambda_2^{j}.$ The following lemma shows that the top terms of a monomial $\boldsymbol{x}^{(s)}\boldsymbol{y}^{(t)}$ cannot be canceled out by any terms from a monomial $\boldsymbol{x}^{(i)}\boldsymbol{y}^{(j)}$of lower bidegree.

\begin{lemma}\label{leading terms}
Suppose $d_2(\boldsymbol{x}^{(i)}\boldsymbol{y}^{(j)})<d_2(\boldsymbol{x}^{(s)}\boldsymbol{y}^{(t)})$
for some $s,t,i,j \in \mathbb{Z}_{\geq 0}.$ Then $\boldsymbol{x}^{(i)}\boldsymbol{y}^{(j)}$ contains no term $\lambda_1^{s+2t}\lambda_2^{s+t}$ and no term $\lambda_1^{s+2t}\lambda_2^t.$
\end{lemma}

\begin{proof}
Suppose the statement is false. By inspection, the highest power of $\lambda_1$ in any of the terms of $\boldsymbol{x}^{(i)}\boldsymbol{y}^{(j)}$ occurs in $\lambda_1^{i+2j}\lambda_2^{i+j}$ and $\lambda_1^{i+2j}\lambda_2^j.$ By assumption, $i+2j \leq s+2t.$ Thus, if $\boldsymbol{x}^{(i)}\boldsymbol{y}^{(j)}$ contains such a term, we must have $i+2j=s+2t.$ Then by assumption, $s+t>i+j\geq j$ and so neither the term $\lambda_1^{i+2j}\lambda_2^{i+j}$ nor the term $\lambda_1^{i+2j}\lambda_2^j$ can be equal to $\lambda_1^{s+2t}\lambda_2^{s+t}.$ For the remaining case, suppose one of the terms is equal to $\lambda_1^{s+2t}\lambda_2^t.$ If $\lambda_1^{i+2j}\lambda_2^j=\lambda_1^{s+2t}\lambda_2^t$ then we must have $j=t$ and $i=s,$ which contradicts our hypothesis. If $\lambda_1^{i+2j}\lambda_2^{i+j}=\lambda_1^{s+2t}\lambda_2^t$ then we must have $i+2j=s+2t$ and $i+j=t.$ This system implies that $s=-i,$ which implies that $i=s=0$ and $j=t,$ and is a contradiction.

\end{proof}

\begin{theorem}\label{uniqueness}

If $S(x,y) \in \mathcal{A}$ is a transparent element, then $q^2$ is a root of unity of some order $n.$ Furthermore, if $n$ is coprime to $3$, then we must have $S(x,y) \in \mathcal{R}[P_n,Q_n],$ where $P_n:=P_n(x,y)$ and $Q_n:=Q_n(x,y).$

\end{theorem}

\begin{proof}

Suppose $S(x,y) \in \mathcal{A}$ is transparent. By Proposition \ref{P,Q basis}, $\{P_iQ_j\}_{i,j \geq 0}$ is a basis for $\mathcal{A}.$ We write $S(x,y)=\sum a_{ij}P_iQ_j.$

Since $S(x,y)$ is transparent, we have by Proposition \ref{transparent alternative} that $S(x^\star, \overline{y})-S(x_\star, \underline{y})=0 \in \mathcal{A}_{1,1}.$

We compute

\begin{align*}
S(x^\star, \overline{y})-S(x_\star, \underline{y})&=\sum a_{ij}(P_i(x^\star, \overline{y})Q_j(x^\star, \overline{y})-P_i(x_\star, \underline{y})Q_j(x_\star, \underline{y}))\\
&=\sum a_{ij}(F^\star(\boldsymbol{x}^{(i)}\boldsymbol{y}^{(j)})-F_\star(\boldsymbol{x}^{(i)}\boldsymbol{y}^{(j)}))\\
&=\sum a_{ij}(F_\star(\tilde{\boldsymbol{x}}^{(i)}\tilde{\boldsymbol{y}}^{(j)})-F_\star(\boldsymbol{x}^{(i)}\boldsymbol{y}^{(j)})))\\
&=F_\star(\sum a_{ij} (\tilde{\boldsymbol{x}}^{(i)}\tilde{\boldsymbol{y}}^{(j)}-\boldsymbol{x}^{(i)}\boldsymbol{y}^{(j)})).\\
\end{align*}

Let $P_{s}Q_{t}$ be the unique term of $S(x,y)$ of maximal bidegree $(s+2t,s+t).$

Using the injectivity of $F_\star$ and considering the leading terms of Lemma \ref{leading terms}, we have

\begin{align*}
a_{st}(q^{4s+6t}-1)&=0\\
a_{st}(q^{2s+6t}-1)&=0,\\
\end{align*}

where $a_{st}(q^{4s+6t}-1)$ is the coefficient of $\lambda_1^{s+2t}\lambda_2^{s+t}$ in $\sum a_{ij} (\tilde{\boldsymbol{x}}^{(i)}\tilde{\boldsymbol{y}}^{(j)}-\boldsymbol{x}^{(i)}\boldsymbol{y}^{(j)})$ and $a_{st}(q^{2s+6t}-1)$ is the coefficient of $\lambda_1^{s+2t}\lambda_2^{t}.$

Since $a_{st} \neq 0$ and $\mathcal{R}$ is an integral domain, we deduce that $q^2$ is a root of unity of some order $n$ such that

\begin{align*}
2s+3t \equiv 0 \mod{n}\\
s+3t \equiv 0 \mod{n}.\\
\end{align*}

We deduce that $s,3t \equiv 0 \mod{n}.$ If $n$ is not divisible by $3$, then we deduce that $s,t \equiv 0 \mod{n}.$ In this case, there exist $k_{s},k_{t} \in \mathbb{Z}_{\geq 0}$ such that $s=k_{s}n$ and $t=k_{t}n.$

By Corollary \ref{power sum corollary}, $P_{s}(x,y)=P_{k_{s}n}(x,y)=P_{k_{s}}(P_n(x,y),Q_n(x,y))$ and $Q_{t}(x,y)=Q_{k_{t}}(P_n(x,y),Q_n(x,y)).$ Thus, the top term of $S(x,y)$ is in $\mathcal{R}[P_n,Q_n].$ Consequently, $S(x,y)-a_{st}P_{s}Q_{t}$ is a transparent element of lower bidegree than $S(x,y).$ A repeated application of this argument shows that $S(x,y) \in \mathcal{R}[P_n,Q_n].$

\end{proof}

\bibliographystyle{amsalpha}
\bibliography{G2.bib}
\end{document}